\documentclass[smallextended]{amsart}
\usepackage[top=1in, bottom=1in, left=1.25in, right=1.25in]{geometry}
\usepackage{mathrsfs}
\usepackage{amssymb}
\usepackage{amsfonts}
\usepackage[linkcolor=red,citecolor=blue]{hyperref}
\usepackage{latexsym}
\usepackage[all,cmtip]{xy}
\usepackage{amsmath}
\usepackage{color,ifsym}
\usepackage{graphicx}
\usepackage{fancybox}
\usepackage{longtable}
\usepackage{lscape}
\usepackage{float}
\usepackage{mathdots}
\usepackage{enumerate}
\usepackage{multirow}

\newtheorem{thm}{Theorem}%[section]  %   number by chapter
\newtheorem{cor}[thm]{Corollary}

\newtheorem{conj}[thm]{Conjecture}

\newtheorem{rem}{{\it Remark}} 
\begin{document}
	\title{Universal hierarchical structure of reducibility of Harish-Chandra parabolic induction}
	\author{CAIHUA LUO}
	\address{Department of Mathematical Sciences, Chalmers University of Technology and the University of Gothenburg, Chalmers tv\"{a}rgata 3, se-412 96 g\"{o}teborg}
	\email{caihua@chalmers.se}
%\institute{Caihua Luo \at
%	Department of Mathematical Sciences, Chalmers University of Technology and the University of Gothenburg, Chalmers Tv\"{a}rgata 3, se-412 96 G\"{o}teborg \\
%	Tel.: +86 15576787967\\
%	\email{caihua@chalmers.se}           %  \\
%	\and
%	S. Author \at
%	second address
%}
	\date{}
\maketitle
	\begin{abstract}
		Given a supercuspidal representation $\sigma$ of a parabolic subgroup $P$ of reductive group $G$, we discover a universal hierarchical structure of reducibility of the parabolic induction $Ind^G_P(\sigma)$, i.e. always irreducible from some Levi-level up. As its applications, we provide a new simple proof of the generic irreducibility property of parabolic induction, and prove Clozel's finiteness conjecture of special exponents under some conditions. Indeed, those conditions are predicted by two conjectures of Shahidi which in some sense are proved for classical groups by Arthur in his monumental book--The Endoscopic Classification of Representations: Orthogonal and Symplectic Groups. At last, naturally, such type simple beautiful structure theorem should be conjectured to hold in general, i.e. if the ``reducibility conditions'' of a general parabolic induction lies in some Levi subgroup, then it is always irreducible from this Levi up.
%			\subclass{22E35 \and 22E50}
%			\keywords{Parabolic induction \and Generic irreducibility property \and Special exponents \and Intertwining operator\and Jacquet module\and Knapp--Stein R-group\and Covering groups} 
	\end{abstract}
%	\maketitle

	% !TeX root = parabolicinduction
\section{Introduction}

In this paper, we study the structure of parabolic induction of admissible representations of reductive groups defined over a non-archimedean local field $F$ of characteristic 0. Of course, the problem goes back to the era of Harish-Chandra and marches on in the century of Langlands for its deep connection with number theory, many great mathematicians have devoted their efforts on this topic (cf. \cite{knappstein,knapp1976classification,langlands,jacquet,howe1976any,bernstein1977induced,speh1980reducibility,kazhdan1987proof,casselman1995introduction,shahidi1990proof,moeglin2002construction,moeglin2003points}).

Let $G$ be a connected reductive group defined over $F$, and $P=MN$ be a parabolic subgroup of $G$ with $M$ its Levi subgroup and $N$ nilpotent radical of $P$. For an irreducible admissible representation $(\sigma, V_\sigma)$ of $M$, a fundamental question originating from Harish-Chandra's theory is to give a reasonable criterion of the reducibility of the parabolic induction $I^G_P(\sigma)$:
\[I^G_P(\sigma)=Ind^G_P(\sigma):=\{f:G\rightarrow V_\sigma \mbox{ smooth }|f(mng)=\delta_P(m)^{\frac{1}{2}}\sigma(m)f(g), \forall m\in M,n\in N,g\in G \}. \]
If $G$ is a finite group, such a question is answered perfectly by Mackey's theory. Indeed, analogous Mackey theory does exist for our $G$ which is the so-called Bernstein--Zelevinsky geometrical lemma (cf. \cite{bernstein1977induced,casselman1995introduction,waldspurger2003formule}), and analogous simple criterion does exist for tempered inductions $I^G_P(\sigma)$, i.e. $\sigma$ is a discrete series representation. Such a criterion is the so-called Knapp--Stein $R$-group theory (cf. \cite{knappstein,knapp1975singular,silberger1978knapp,luo2017R}). But there is an essential obstruction for general $\sigma$: the restriction functor, i.e. Jacquet functor does not preserve unitarity as opposed to the finite group case. If $\sigma$ is a supercuspidal representation, i.e. does not arise from proper parabolic inductions, a simple criterion has been discovered recently which originates from Muller's criterion for principal series, i.e. $\sigma$ is a character (cf. \cite{muller1979integrales,kato1982irreducible,luo2018muller}). For general $\sigma$ and general $G$, there is no essential progress (to my best knowledge). Inspired by our Muller type irreducibility criterion for $\sigma$ supercuspidal, we give a first structural answer for a large class of $\sigma$. Denote by $\rho$ the supercuspidal support of $\sigma$ on $P_0=M_0N_0\subset P=MN$, i.e. $\sigma\in JH(I^M_{M\cap P_0}(\rho))$ the set of Jordan--H\"{o}lder constituents of $I^M_{M\cap P_0}(\rho)$, then our universal hierarchical irreducibility criterion is as follows: please see the context for the detail,
\begin{thm}(cf. Theorem \ref{prod}) Assume the reducibility conditions given by Muller type irreducibility criterion for $I^G_{P_0}(\rho)$ lie in $M$, then
	$I^G_P(\tau)$ is always irreducible for any $\tau\in JH(I^M_{M\cap P_0}(\rho))$.
\end{thm} 
Such a criterion exists for tempered inductions which is a key ingredient of proving Howe's finiteness conjecture (cf. \cite{clozel1989orbital,luo2017howe}). Indeed, Clozel also proposed a beautiful conjecture as a key ingredient in his first attempt to prove Howe's finiteness conjecture (see \cite{clozel1985conjecture}). As an application of the above theorem, we serve you a simple intuitive proof of such a beautiful conjecture under the assumption that it holds for co-rank one cases.

Let us take a brief look at the main idea of the proof of Theorem \ref{prod} for regular $\rho$, i.e. no non-trivial Weyl element fixes it. Our argument is quite classical. It relies heavily on the Bernstein--Zelevinsky geometrical lemma and the exactness of Jacquet functor (cf. \cite{bernstein1977induced,casselman1995introduction,silberger2015introduction,waldspurger2003formule}): the constituents of the generalized principal series $I^G_{P_0}(\rho)$ are parametrized by the partition of the set of the relative Weyl group $W_{M_0}:=N_G(M_0)/M_0$. If our reducibility conditions of $I^G_{P_0}(\rho)$ given by our Muller type irreducibility criterion lie in some standard Levi subgroup $M$, then it gives rise to a partition of $W_{M_0}^M:=N_M(M_0)/M_0\subset W_{M_0}$. Then to show the irreducibility of $I^G_P(\sigma)$ for any $\sigma\in JH(I^M_{M\cap P_0}(\rho))$, the novelty of our argument is to find a good set $S$ of representatives of the quotient $W_{M_0}^M\backslash W_{M_0}$, such that    
\[I^G_P(\rho^{w_1})\simeq I^G_P(\rho^{w_1w}) \]
for any $w_1\in W^M_{M_0}$ and $w\in S$.

It seems that our argument is quite restrictive. But it also seems that this may be the only universal argument we could come up with nowadays which has been practiced intelligently by Bernstein--Zelevinsky, Silberger, Jantzen, Moeglin, the Tadi{\'c} school etc (please refer to \cite{bernstein1977induced,jantzen1997supports,moeglin2002construction,tadic1998regular,lapid2017some} for a glimpse). Inspired by a conjectural Muller type irreducibility criterion for general parabolic inductions (see \cite[Conjecture 4.4]{luo2018muller}), a detailed study of the structure of Jacquet modules of a general parabolic induction may give us hope to prove the following conjectural universal irreducibility structure (Please see the context for the details):
\begin{conj}(see Conjecture \ref{conj})
	Given a parabolic induction $I^G_{P=MN}(\sigma)$ with $\sigma$ an irreducible admissible representation of $M$, if the ``reducibility conditions'' of $I^G_P(\sigma)$ lie in some standard Levi subgroup $L$ of a parabolic subgroup $Q=LV\subset P=MN$, then $I^G_Q(\tau)$ is always irreducible for any $\tau\in JH(I^L_{L\cap P}(\sigma))$.
\end{conj}
\begin{rem}
	At first glance, Conjecture \ref{conj} may look meaningless and ridiculous. But some supportive examples of a conjectural Muller type irreducibility criterion for general parabolic induction, i.e. \cite[Conjecture 4.4]{luo2018muller} could guide us to a right direction. Those examples are parabolic induction representations inducing from essentially discrete series and standard modules which many mathematicians have investigated since the era of Harish-Chandra.
\end{rem}
Inspired by Goldberg and Jantzen's product formulas for quasi-split classical groups (cf. \cite{goldberg1994reducibility,jantzen1997supports,jantzen2005duality}), another direction we could cook up with the help of our new notion of $R$-group is a Goldberg--Jantzen type product formula as follows: (Please see the context for the details)
\begin{conj}(see Conjecture \ref{gjconj})
	Let $\nu\in \mathfrak{a}_M^*$ and $\sigma$ be a unitary supercuspidal representation of $M$. For the generalized principal series representation $I^G_P(\nu,\sigma):=Ind^G_P(\sigma\otimes\nu)$ of $G$, there is a one-one correspondence:
	\[JH(I^G_P(\nu,\sigma))\leftrightarrows \prod_i JH(I^{M_i}_{M_i\cap P}(\nu,\sigma)). \]
\end{conj}
Here $\Phi_\sigma:=\{\alpha\in\Phi_M:~w_\alpha.\sigma=\sigma \}$ is a root system and decomposes into irreducible pieces, i.e. $\Phi_\sigma=\sqcup_i\Phi_{\sigma,i}$. Each irreducible root system $\Phi_{\sigma,i}$ gives rise to a Levi subgroup $M_i\supset M$ which is defined by
\[M_i:=C_G((\bigcap\limits_{\alpha\in\Phi_{\sigma,i}}Ker(\alpha) )^0). \]
In the meantime we assume that $R_\sigma$ associated to $\sigma$ decomposes into a product of small pieces in the same pattern as does $\Phi_\sigma$, i.e.
\[R_\sigma=\prod_i R_{\sigma,i} \]
with $R_{\sigma,i}$ a subgroup of the relative Weyl group $W^{M_i}_M=N_{M_i}(M)/M$ of $M$ in $M_i$.

As a by-product of Theorem \ref{prod}, we could reduce Clozel's finiteness conjecture of special exponents to co-rank one cases. There is no harm to assume that $G$ is of compact center. For a discrete series representation $\sigma$ of $G$, we denote its associated supercuspidal support to be $\rho$ which is a supercuspidal representation of some Levi subgroup $M$ of $G$. Denote by $\omega_\rho$ the unramified part of the central character of $\rho$, i.e. $\omega_\rho\in \mathfrak{a}_{M,\mathbb{C}}^\star:=Hom_F(M,\mathbb{G}_m)\otimes \mathbb{C}$. Such a character is called a \underline{special exponent}. Clozel conjectured that the set of special exponents is finite (see \cite{clozel1985conjecture,clozel1989orbital}).

\begin{thm}(cf. Theorem \ref{clozel})
	The set of special exponents is finite provided it is true for co-rank one cases.
\end{thm}
Indeed, the generic co-rank one case is a result of the profound Langlands--Shahidi theory, and the general co-rank one case follows from two conjectures of Shahidi (cf. \cite[Conjectures 9.2 \& 9.4]{shahidi1990proof}). Those two conjectures for classical groups in some sense are by-products of Arthur's standard model argument in his monumental book \cite{arthur2013endoscopic}. The main idea of the proof of Clozel's finiteness conjecture is as follows:
\begin{enumerate}[(i)]
	\item A full induced representation $I^G_P(\sigma)$ can never be discrete series.
	\item An invertible matrix has only one solution.
\end{enumerate}
To illustrate the simple ideas, let us take a look at the real parts of special exponents, if $I^G_P(\rho)$ contains a discrete series subquotient, then the reducibility conditions must generate the whole vector space $\mathfrak{a}^\star_{M,\mathbb{C}}$. Otherwise, there exists a proper parabolic subgroup $Q=LV\supset P$ such that the reducibility conditions lie in $L$, then our universal hierarchical irreducibility criterion implies that $I^G_Q(\tau)$ is always irreducible for any $\tau\in JH(I^L_{L\cap P}(\rho))$, whence any constituent of $I^G_P(\rho)$ can never be discrete series. Contradiction. Thus $\mathfrak{a}^\star_{M,\mathbb{C}}$ is generated by the reducibility conditions. Indeed, we recently learned that such a claim is also a corollary of an old result of Harish-Chandra (cf. \cite[Theorem 5.4.5.7]{silberger2015introduction}, \cite[Theorem 3.9.1]{silberger1981discrete} or \cite[Corollary 8.7]{heiermann2004decomposition}). Under the assumption that the set of special exponents is finite for co-rank one cases, those reducibility conditions form a finite set of hyperplanes in $\mathfrak{a}^\star_{M,\mathbb{C}}$, therefore the finiteness for high rank cases follows easily from the fact that there exists only one solution for an invertible matrix. At last, we would like to mention that Clozel proposed an analogous strong global conjecture of finiteness of poles of Eisenstein series constructed from cusp forms which would simplify Arthur's trace formula machine significantly. We hope that the simple intuitive proof of Clozel's local finiteness conjecture could shed some light on his global conjecture in our future work.

We end the introduction by recalling briefly the structure of the paper. In Section 2, we prepare some necessary notation. In Section 3, we state and prove the universal hierarchical irreducibility criterion for Harish-Chandra parabolic inductions. In the end, as its first application, a new understanding of the generic irreducibility property of parabolic inductions is provided. In Section 4, with the aide of the universal hierarchical irreducibility criterion in the previous section, a simple intuitive proof of Clozel's finiteness conjecture is served under the assumption that it holds for the co-rank one case.    
%		\paragraph*{\textbf{Acknowledgements}}
%		I am much indebted to Professor Wee Teck Gan for his guidance and numerous discussions on various topics. I would like to thank Professor Dipendra Prasad for helpful conservation, and thank Maxim Gurevich for his seminar talk on a conjectural criterion of the irreducibility of parabolic inductions for $GL_n$ in the National University of Singapore which rekindles our enthusiasm to explore the mysterious internal structure of parabolic induction.  
   % !TeX root = parabolicinduction
\section{Preliminaries}
Let $G$ be a connected reductive group defined over a non-archimedean local field $F$ of characteristic 0. Denote by $|-|_F$ the absolute value, by $\mathfrak{w}$ the uniformizer and by $q$ the cardinality of the residue field of $F$. Fix a minimal parabolic subgroup $B=TU$ of $G$ with $T$ a minimal Levi subgroup and $U$ a maximal unipotent subgroup of $G$, and let $P=MN\supset B=TU$ be a standard parabolic subgroup of $G$ with $M$ the Levi subgroup and $N$ the unipotent radical.

\subsection{Structure theory}Let $X(M)_F$ be the group of $F$-rational characters of $M$, and set 
\[\mathfrak{a}_M=Hom(X(M)_F,\mathbb{R}),\qquad\mathfrak{a}^\star_{M,\mathbb{C}}=\mathfrak{a}^\star_M\otimes_\mathbb{R} \mathbb{C}, \]
where
\[\mathfrak{a}^\star_M=X(M)_F\otimes_\mathbb{Z}\mathbb{R} \]
denotes the dual of $\mathfrak{a}_M$. Recall that the Harish-Chandra homomorphism $H_P:M\longrightarrow\mathfrak{a}_M$ is defined by
\[q^{\left< \chi,H_P(m)\right>}=|\chi(m)|_F \] 
for all $\chi\in X(M)_F$.

Next, let $\Phi$ be the root system of $G$ with respect to $T$, and $\Delta$ be the set of simple roots determined by $U$. For $\alpha\in \Phi$, we denote by $\alpha^\vee$ the associated coroot, and by $w_\alpha$ the associated reflection in the Weyl group $W=W^G$ of $T$ in $G$ with
\[W:=N_G(T)/T=\left<w_\alpha:~\alpha\in\Phi\right>. \]
Denote by $w_0^G$ the longest Weyl element in $W$, and similarly by $w_0^M$ the longest Weyl element in the Weyl group $W^M:=N_M(T)/T$ of a Levi subgroup $M$. 

Likewise, we denote by $\Phi_M$ (resp. $\Phi_M^L$) the reduced relative root system of $M$ in $G$ (resp. the Levi subgroup $M\subset L$), by $\Delta_M$ the set of relative simple roots determined by $N$ and by $W_M:=N_G(M)/M$ (resp. $W_M^L$) the relative Weyl group of $M$ in $G$ (resp. $L$). In general, a relative reflection $\omega_\alpha:=w_0^{M_\alpha}w_0^M$ with respect to a relative root $\alpha$ does not preserve our Levi subgroup $M$. Denote by $\Phi^0_M$ (resp. $\Phi^{L,0}_M$) the set of those relative roots which contribute reflections in $W_M$ (resp. $W_M^L$). It is easy to see that $W_M$ preserves $\Phi_M$, and further $\Phi_M^0$ as well, as $\omega_{w.\alpha}=w\omega_\alpha w^{-1}$. Note that $W_M$ (resp. $W_M^{L,0}$) in general is larger than $W_M^0$ (resp. $W_M^{L,0}$) the one generated by those relative reflections in $G$ (resp. $L$). Denote by $\Phi_M(P)$ the set of reduced roots of $M$ in $P$.

Recall that the canonical pairing $$\left<-,-\right>:~\mathfrak{a}^\star_M\times \mathfrak{a}_M\longrightarrow\mathbb{R}$$ suggests that each  $\alpha\in \Phi_M$ will enjoy a one parameter subgroup $H_{\alpha^\vee}(F^\times)$ of $M$ satisfying: for $x\in F^\times$ and $\beta\in \mathfrak{a}^\star_M$,
\[\beta(H_{\alpha^\vee}(x))=x^{\left<\beta,\alpha^\vee\right>}. \]

\subsection{Parabolic induction and Jacquet module}For $P=MN$ a parabolic subgroup of $G$ and an admissible representation $(\sigma,V_\sigma)~ (resp.~(\pi,V_\pi))$ of $M~(resp.~G)$, we have the following normalized parabolic induction of $P$ to $G$ which is a representation of $G$
\[I_P^G(\sigma)=Ind_P^G(\sigma):=\{\mbox{smooth }f:G\rightarrow V_\sigma|~f(nmg)=\delta_P(m)^{1/2}\sigma(m)f(g), \forall n\in N, m\in M~and~g\in G\} \]
with $\delta_P$ stands for the modulus character of $P$, i.e., denote by $\mathfrak{n}$ the Lie algebra of $N$,
\[\delta_P(nm)=|det~Ad_\mathfrak{n}(m)|_F, \]
and the normalized Jacquet module $J_M(\pi)$ with respect to $P$ which is a representation of $M$
\[\pi_N:=V_\pi/\left<\pi(n)e-e:~n\in N,e\in V_\pi\right>. \] 
Given an irreducible admissible representation $\sigma$ of $M$ and $\nu\in \mathfrak{a}^\star_{M}$, let $I(\nu,\sigma)$ be the representation of $G$ induced from $\sigma$ and $\nu$ as follows:
\[I(\nu,\sigma)=Ind_P^G(\sigma\otimes q^{\left<\nu,H_P(-)\right>}) .\]
We denote by $JH(I_P^G(\sigma))$ the set of Jordan--H\"{o}lder constituents of the parabolic induction $I^G_P(\sigma)$, and define the action of $w\in W_M$ on representations $\sigma$ of $M$ to be $w.\sigma=\sigma\circ Ad(w)^{-1}$ and $\sigma^w=\sigma\circ Ad(w)$.

\subsection{$R$-group}
In \cite{muller1979integrales}, for a principal series $I(\lambda)$ of $G$, she defines a subgroup $W_\lambda^1$ of the Weyl group $W$ governing the reducibility of the ``unitary'' part of principal series on the Levi level, which is indeed the Knapp--Stein $R$-group as follows (cf. \cite{winarsky1978reducibility,keys1982decomposition}), 
\begin{align*}
\Phi_{\lambda}^0&:=\{\alpha\in \Phi:~\lambda_\alpha=Id \},\\
W_{\lambda}^0&:=\left<w_\alpha:~\alpha\in \Phi_\lambda^0 \right>,\\
W^1_\lambda&:=\{w\in W_\lambda:~w.(\Phi_\lambda^0)^+>0 \},\\
W_\lambda&:=\{w\in W:~w.\lambda=\lambda \}.
\end{align*} 
In view of \cite[Lemma I.1.8]{waldspurger2003formule}, one has
\[W_\lambda=W_\lambda^0\rtimes W_\lambda^1. \]
Following the Knapp--Stein R-group theory (cf. \cite{silberger2015introduction}), we denote by $R_\lambda$ the subgroup $W_\lambda^1$.

Likewise, for generalized principal series $I^G_P(\sigma)$ (cf. \cite{luo2018muller}),
\begin{align*}
\Phi_{\sigma}^0&:=\{\alpha\in \Phi_M^0:~w_\alpha.\sigma=\sigma \},\\
W_{\sigma}^0&:=\left<w_\alpha:~\alpha\in \Phi_{\sigma}^0 \right>,\\
W^1_{\sigma}&:=\{w\in W_{\sigma}:~w.(\Phi_{\sigma}^0)^+>0 \},\\
W_{\sigma}&:=\{w\in W_M:~w.\sigma=\sigma \}.
\end{align*}
Via \cite[Lemma I.1.8]{waldspurger2003formule}, we have
\[W_{\sigma}=W_{\sigma}^0\rtimes W_{\sigma}^1, \]
and we denote $R_{\sigma}$ to be $W_{\sigma}^1$ following tradition, but it is not the exact Knapp--Stein $R$-group in the sense of Silberger.

\subsection{Special exponent} There is no harm to assume that $G$ is of compact center. For a discrete series representation $\pi$ of $G$, we denote its associated supercuspidal support to be $\sigma$ which is a supercuspidal representation of some Levi subgroup $M$ of $G$. Denote by $\omega_\sigma$ the unramified part of the central character of $\sigma$, i.e. $\omega_\sigma\in \mathfrak{a}_{M,\mathbb{C}}^\star$. Such a character is called a \underline{special exponent}.

   % !TeX root = parabolicinduction
\section{Universal Hierarchical Structure of Reducibility}
\subsection{A product formula}
In this subsection, we prove a key observation of the decomposition of parabolic induction which opens a gate to understand some of the classical results/conjectures, for example the generic irreducibility property of parabolic induction and Clozel's finiteness conjecture of special exponents.

Recall that $G$ is a connected reductive group defined over $F$ with the set of simple roots $\Delta$, $P=MN$ is a standard parabolic subgroup of $G$ associated to $\Theta_M\subset \Delta$ and $\sigma$ is a supercuspidal representation of $M$ (not necessary unitary), one forms a parabolic induction $I_P^G(\sigma)$. Then our ``product formula'' is designed to ask the following question
\[\mbox{When does the reducibility of $I^G_P(\sigma)$ only happen on the Levi-level?\tag*{$(\star)$}} \]
i.e.
\[\mbox{What is a reasonble condition for the irreducibility of $I^G_{Q=LV}(\tau)$ for all $\tau\in JH(I_{L\cap P}^L(\sigma))$?} \]
The answer traces back to a beautiful theorem of I. Muller \cite{muller1979integrales,kato1982irreducible} which provides a natural criterion of the irreducibility of principal series, and its generalized version for generalized principal series \cite{luo2018muller} using the Knapp--Stein $R$-group and co-rank one reducibility. As the irreducibility is governed by the Knapp--Stein $R$-group and the co-rank one reducibility, a natural candidate for ($\star$) is to assume that those governing conditions occur only on the Levi-level. To be more precise, let $Q=LV$ be a standard parabolic subgroup associated to $\Theta_L$ with $\Theta_M\subset\Theta_L\subset \Delta$, then our working assumption for $(\star)$ is as follows.

{\bf \underline{Working Hypothesis}}:
\leavevmode
\begin{enumerate}[(i)]
	\item (\underline{Rank-one reducibility}) The co-rank one reducibility only occurs within $L$, i.e.
	\[I_M^{M_\alpha}(\sigma) \mbox{ is reducible only for some }\alpha\in \Phi_M^{L},~i.e.~\alpha\in \Phi_M^{L,0}\mbox{ (cf. \cite[Theorem 7.1.4]{casselman1995introduction})}. \] 
	\item (\underline{$R$-group}) The $R$-group $R_\sigma$ associated to $\sigma$ is a subgroup of $W^L_M$, i.e.
	\[R_\sigma\subset W^L_M. \]
\end{enumerate}
Under those hypothesis, via the Jacquet module machine, the confirmation of $(\star)$ results from the associativity property of intertwining operators and the following observation/fact (cf. \cite{bernstein1977induced,casselman1995introduction,moeglin_waldspurger_1995,waldspurger2003formule}).
\begin{enumerate}[(i)]
	\item (Bernstein--Zelevinsky geometrical lemma)
	\[J_M( I_P^G(\sigma))=\sum\limits_{w\in W_M}\sigma^w. \]
	\item (Bruhat--Tits decomposition)
	\[W_M=W_M(L)W^L_M, \tag{$\star\star$}\]
	where $W_M(L)$ is defined as follows:
	\[W_M(L):=\{w\in W_M:~w.(\Phi_M^{L,0})^+>0 \}/W_M^{L,1}, \]
	here $W_M^{L,1}$ is defined as in \cite[Lemma 2.1]{luo2018muller}), i.e.
	\[W_M^{L,1}:=\{w\in W_M^L:w.(\Phi_M^{L,0})^+>0 \}. \]
\end{enumerate}
To be precise,
\begin{thm}(Universal hierarchical irreducibility criterion)\label{prod}
	Keep the notions as above. Under the above {\bf Working Hypothesis}, we have the following cardinality equality
	\[\#JH(I^G_P(\sigma))=\#JH(I_{L\cap P}^L(\sigma)). \]
	Indeed, this is equivalent to saying that 
	\[I^G_Q(\tau)\mbox{ is always irreducible for any }\tau \in JH(I^L_{L\cap P}(\sigma)). \]
\end{thm}
Before turning to the proof, let us first prove the Bruhat-Tits decomposition, i.e. $(\star\star)$, which is a generalization of \cite[Lemma 1.1.2]{casselman1995introduction} for $M=T$, in what follows.
\begin{proof}[Proof of $(\star\star)$]
	First note that (cf. \cite[Lemma 2.1]{luo2018muller}) 
	\[W_M:=W_M^0\rtimes W_M^1 \]
	and 
	\[W_M^L:=W_M^{L,0}\rtimes W_M^{L,1}, \]
	with
	\[W_M^{1}:=\{w\in W_M:w.(\Phi_M^{0})^+>0 \}=\{w\in W_M:w.(\Phi_M^{0})^+=(\Phi_M^{0})^+ \}. \]
	So 
	\[W_M/W_M^L=\{w\in W_M:w.(\Phi_M^{L,0})^+>0 \}/W_M^{L,1}. \]
\end{proof}
\begin{proof}[Proof of Theorem \ref{prod}]
	Note that the decomposition of $I^G_P(\sigma)$ is a partition of $W_M$. Recall that $R_\sigma$ is in general not the exact R-group in the sense of Knapp--Stein, as it is defined by
	\[W_\sigma=W_\sigma^0\rtimes R_\sigma, \]
	where
	\[W_\sigma:=\{w\in W_M:~w.\sigma=\sigma \},\]
	and 
	\[W_\sigma^0:=\left<w_\alpha:~w_\alpha.\sigma=\sigma,~\alpha\in\Phi_M  \right>. \]
	But reducibility coming from $W_\sigma^0$ has been taken care of by the co-rank one reducibility condition which only occurs within $L$ by the assumption. On the other hand, the Knapp-- Stein $R$-group theory helps us control the multiplicity issue.
	
	Therefore it reduces to show that the non-zero intertwining operator $A(w,\sigma)$ associated to $w\in W_M(L)$ is an isomorphism, i.e.
	\[A(w,\sigma):~I_P^G(\sigma)\stackrel{\sim}{\longrightarrow} I_P^G(\sigma^w). \]
	Recall that $A(w,\sigma)$ is defined as follows:
	\[J_{P|P^w}(\sigma^w)\circ \lambda(w):~I^G_P(\sigma)\longrightarrow I^G_{P^w}(\sigma^w)\longrightarrow I^G_P(\sigma^w). \]
	Thus the above isomorphism claim follows from the associativity property of intertwining operators (cf.\cite[IV.3 (4)]{waldspurger2003formule} or \cite[Lemma 3.5]{luo2018R}), i.e.
	\[J_{P|P^w}(\sigma^w)J_{P^w|P}(\sigma)=\prod j_\alpha(\sigma)J_{P|P}(\sigma), \]
	where $\alpha$ runs over
	\[\Phi_M(P)\bigcap \Phi_M(\overline{P^w}) \]
	with $\overline{P^w}$ the opposite parabolic subgroup of $P^w$, and $\Phi_M(P)$ (resp. $\Phi_M(\overline{P^w})$) is the set of restricted roots of $M$ in $P$ (resp. $\overline{P^w}$).
	
	Note that for $\alpha\in \Phi_M(P)-\Phi_M^0$, the associated co-rank one induction is always irreducible and $j_\alpha(\sigma)\neq 0,\infty$ (cf. \cite[Corollary 1.8]{silberger1980special}). 
	
	Note also that for $\alpha\in \Phi_M^{0}-\Phi_M^{L,0}$, the associated co-rank one induction is always irreducible by our Working Hypothesis, which implies that either $j_\alpha(\sigma)\neq 0,\infty$ for non-unitary induction (cf. \cite[Corollary 1.8]{silberger1980special}), or $j_\alpha(\sigma)$ has a pole of order 2 (cf. \cite[Proposition 2]{savin2007}). For the latter case, one can take the residue to get an isomorphism. 
	
	Therefore one only needs to consider those $j_\alpha(\sigma)$ with $\alpha\in \Phi_M^{L,0}$. 
	
	For those $\alpha\in \Phi_M^{L,0}$, we have 
	\[w.(\Phi_M^{L,0})^+>0, \]
	so
	\[\Phi_M(P)\bigcap \Phi_M(\overline{P^w})\bigcap \Phi_M^{L,0}=\emptyset.  \]
	Thus $A(w,\sigma)$ is an isomorphism.
\end{proof}
It seems that our proof of Theorem \ref{prod} is quite restrictive. But it also seems that this may be the only universal argument we could come up with nowadays which has since been practiced intelligently by Bernstein--Zelevinsky, Silberger, Jantzen, Moeglin, the Tadi{\'c} school etc (please refer to \cite{bernstein1977induced,jantzen1997supports,moeglin2002construction,tadic1998regular,lapid2017some} for a glimpse). Inspired by a conjectural Muller type irreducibility criterion for general parabolic inductions (see \cite[Conjecture 4.4]{luo2018muller}), a detailed study of the structure of Jacquet modules of a general parabolic induction may give us hope to prove the following conjectural universal irreducibility structure:
\begin{conj}\label{conj}
	Given a parabolic induction $I^G_{P=MN}(\sigma)$ with $\sigma$ an irreducible admissible representation of $M$, if the reducibility conditions of $I^G_P(\sigma)$ lie in some standard Levi subgroup $L$ of a parabolic subgroup $Q=LV\subset P=MN$, then $I^G_Q(\tau)$ is always irreducible for any $\tau\in JH(I^L_{L\cap P}(\sigma))$.
\end{conj}
\begin{rem}
	At first glance, Conjecture \ref{conj} may look meaningless and ridiculous. But some supportive examples of a conjectural Muller type irreducibility criterion for general parabolic induction, i.e. \cite[Conjecture 4.4]{luo2018muller} could guide us to a right direction. Those examples are parabolic induction representations inducing from essentially discrete series and standard modules which many mathematicians have investigated since the era of Harish-Chandra.
\end{rem}
Denote by $\Theta_Q$ the associated subset of $\Delta$ which determines the parabolic subgroup $Q=LV\supset P=MN$ of $G$. Explicitly, we decompose $\Theta_L=\Theta_1\sqcup\cdots \sqcup \Theta_t $ into irreducible pieces, and accordingly $\Theta_M=\Theta_1^M\sqcup \cdots \sqcup \Theta_t^M$. Assume that $R_\sigma$ decomposes into $R_\sigma=R_1\times \cdots \times R_t$ with respect to the decomposition of $\Theta_L$, and a similar decomposition pattern holds for the co-rank one reducibility, i.e. co-rank one reducibility only occurs within $P_{\Theta_i}=M_{\Theta_i}N_{\Theta_i}$ for $1\leq i\leq t$. Then we have  
\begin{cor}[Product formula]\label{prod1}
	\[\#(JH(I^G_P(\sigma)))=\prod_{i=1}^{t}\#(JH(I_{M_{\Theta_i^M}}^{M_{\Theta_i}}(\sigma))).\]
\end{cor}
\begin{rem}
	Very recently, we learned that Jantzen has a beautiful product formula for split $Sp_{2n}$, $SO_{2n+1}$ and $O_{2n}$ which in some sense originates from Goldberg's product formula for tempered inductions (cf. \cite{jantzen1997supports,jantzen2005duality,goldberg1994reducibility}). Note that our new $R$-group is always trivial for those groups in Goldberg and Jantzen's works. Inspired by their beautiful theorems, we would like to investigate what kind of general product formula we could prove under our new notion of $R$-group in our future work.
\end{rem}
As an instance, one version of Goldberg--Jantzen type product formula is as follows. Consider the tempered generalized principal series $I^G_P(\sigma)$ with $\sigma$ unitary supercuspidal representation of $M$, we know that $\Phi_\sigma:=\{\alpha\in\Phi_M:~w_\alpha.\sigma=\sigma \}$ is a root system, may be reducible. Decomposing $\Phi_\sigma$ into irreducible pieces, i.e. $\Phi_\sigma=\sqcup_i\Phi_{\sigma,i}$. Each irreducible root system $\Phi_{\sigma,i}$ gives rise to a Levi subgroup $M_i\supset M$ which is defined by
\[M_i:=C_G((\bigcap\limits_{\alpha\in\Phi_{\sigma,i}}Ker(\alpha) )^0). \]
Assume that our new $R$-group $R_\sigma$ associated to $\sigma$ decomposes into a product of small pieces in the same pattern as does $\Phi_\sigma$, i.e.
\[R_\sigma=\prod_i R_{\sigma,i} \]
with $R_{\sigma,i}$ a subgroup of the relative Weyl group $W^{M_i}_M=N_{M_i}(M)/M$ of $M$ in $M_i$, then we have
\begin{conj}(Goldberg--Jantzen type product formula)\label{gjconj}
	Keep the notions as above. For $\nu\in \mathfrak{a}_M^*$, there is a one-one correspondence:
	\[JH(I^G_P(\nu,\sigma))\leftrightarrows \prod_i JH(I^{M_i}_{M_i\cap P}(\nu,\sigma)). \]
\end{conj} 

\subsection{Generic irreducibility property of parabolic induction}
In this subsection, we provide a new simple proof of the generic irreducibility property of parabolic inductions which plays an essential role in Harish-Chandra's Plancherel formula \cite[IV.3]{waldspurger2003formule}.

Given an irreducible smooth representation $\sigma$ of the Levi subgroup $M$ of a parabolic subgroup $P=MN$ in reductive group $G$, we form a family of normalized parabolic induction representations $I(\nu,\sigma)=Ind^G_P(\sigma\otimes \nu)$ of $G$, where $\nu$ varies in $\mathfrak{a}_{M,\mathbb{C}}^\star$. The generic irreducibility property says that 
\begin{thm}(Generic Irreducibility Theorem cf. \cite{sauvageot})
	Keep the notions as above, we have
	\[Irred_\sigma:=\{\nu\in \mathfrak{a}_{M,\mathbb{C}}^\star:~I(\nu,\sigma) \mbox{ is irreducible} \}\mbox{ is a non-trivial Zariski open subset.} \]
\end{thm}
\begin{proof}
	In view of the Langlands classification theorem, it reduces to the case where $\pi$ is supercuspidal. Based on Theorem \ref{prod} (or Muller type irreducibility criterion \cite{luo2018muller}), the reducibility conditions are controlled by co-rank one reducibility and $R$-group. As there exists a unique reducibility point for the co-rank one case (cf. \cite[Lemma 1.2 \& 1.3]{silberger1980special}), thus it reduces to consider the $R$-group condition. Note that the $R$-group is a subgroup of $W_{\sigma_\nu}:=\{w\in W_M:~w.(\sigma\otimes\nu)=\sigma\otimes \nu \}$, whence the non-trivial set $Irred_\sigma$ is Zariski open.  
\end{proof}  
   % !TeX root = parabolicinduction
\section{Clozel's Finiteness Conjecture Of Special Exponents}
In this section, let us start with quoting Clozel's remark on Clozel's finiteness conjecture of special exponents in his second paper on Howe's finiteness conjecture \cite[P. 3]{clozel1989orbital} as follows:

\textit{``We would like to finish the introduction with the remark that the stronger conjecture still retains some interest, although we do not know any obvious application. Here the analogy with the theory of automorphic forms is interesting. In the automorphic case, the analogue of the finiteness assumption about exponents would be the fact that the poles of Eisenstein series constructed from cusp forms on a given parabolic subgroup lie in a fixed, finite set independent of the inducing cusp form. This is a very strong conjecture, unknown even for $GL(n)$, although it would result from the conjectural description of the residues stated by Jacquet in [11]. If this conjecture was true, that would trivially imply that the operator defined by a smooth, $K$-finite function on the adelic group acting on the discrete spectrum is trace-class.''}

As an application of Theorem \ref{prod}, under some conditions, we prove Clozel's finiteness conjecture of special exponents proposed in \cite{clozel1985conjecture} which plays an essential role in Clozel's firt attempt to proving Howe's finiteness conjecture. Note that Clozel's finiteness conjecture may be checked directly for classical groups from Moeglin--Tadic's work on the classification of discrete series (cf. \cite{moeglin2002construction,moeglin2007classification}). As the conjecture is much of a quantitative result, it should be proved with little forces, instead of resorting to such a big stick. Indeed, our proof is quite natural and may shed some lights on the global analogy. In what follows, we first recall some notions.

There is no harm to assume that $G$ is of compact center. Recall that for a discrete series representation $\pi$ of $G$, we write its associated supercuspidal support as $\sigma$ which is a supercuspidal representation of some Levi subgroup $M$ of $G$. Denote by $\omega_\sigma$ the unramified part of the central character of $\sigma$, i.e. $\omega_\sigma\in \mathfrak{a}_{M,\mathbb{C}}^\star$. Such a character is called a \underline{special exponent}.
\begin{thm}(Clozel's finiteness conjecture)\label{clozel}
	The set of special exponents is finite provided it holds for co-rank one cases.
\end{thm}

Before turning to the proof, let us first talk about the main idea.

Under the {\bf Induction Assumption}, i.e.
\[\mbox{Clozel's finiteness conjecture holds for the co-rank one case.}\]
Our proof of Clozel's finiteness conjecture for the general case rests on the following two novel observations:
\begin{enumerate}[(i)]
	\item Theorem \ref{prod}, or ``Product Formula''.
	\item Irreducible induced representation can never be a discrete series.
\end{enumerate}
Roughly speaking, with the help of Muller type theorem of generalized principal series, one knows that the decomposition of $I_P^G(\sigma)$ is governed by the co-rank one reducibility and the $R$-group $R_\sigma$. On the other hand, one knows that a full induced representation can not be a discrete series. In view of Theorem \ref{prod}, thus in order to ensure $\omega_\sigma$ is a special exponent, those roots associated to the co-rank one reducibility and $R_\sigma$ must generate the whole space $\mathfrak{a}_{M,\mathbb{C}}^\star$. Then the conjecture follows from an easy fact of linear algebra, i.e. an invertible matrix has only one solution.

To be more precise, let $P_\Theta=M_\Theta N$ be a standard parabolic subroup of $G$ with $\Theta\subset \Delta$, and let $\sigma$ be a supercuspidal representation of $M_\Theta$. Decomposing $\Theta$ into irreducible pieces
\[\Theta=\Theta_1\sqcup\cdots\sqcup\Theta_n \]
As $W_{M_\Theta}$ acts on $M_\Theta$, then it preserves the decomposition up to sign, so does $R_\sigma$, i.e. preserving
\[\pm\Theta=\pm\Theta_1\sqcup\cdots\sqcup\pm\Theta_n.\] 
Under the action of $R_\sigma$ on $\pm\Theta$, we have a new decomposition of $\pm\Theta$ into irreducible pairs, i.e.
\[R_\sigma\hookrightarrow S_{\pm n}, \]
where $S_{\pm n}$ is the ``pseudo''-permutation group, i.e.
\[S_{\pm n}:=\{\left( a_{i_1}\cdots  a_{i_k}\right):~ a_{i_j}\in \{\pm 1,\cdots,\pm n\} \}/\pm. \]
Here $\pm $-equivalence means that
\[\left(a_{i_1}\cdots a_{i_k}\right)=\left((- a_{i_1})\cdots (-a_{i_k})\right) .\]
For each simple permutation $s=\left(a_{i_1}\cdots  a_{i_k}\right)$, we define the associated roots as, up to scalar,	
\[\Phi_s:=\{e_{i_j}-e_{i_l}:~1\leq j<l\leq k \}, \]
where $e_{i_j}$ is the component character on $\Theta_{i_j}$.

\begin{proof}
	It suffices to prove the finiteness of special exponents for only one parabolic subgroup, like $P_\Theta=M_\Theta N$ with $\sigma$ supercuspidal representations of $M_\Theta$ and $\Theta\subset \Delta$.
	
	Considering the set $\Phi_\sigma$ of roots associated to the co-rank one reducibility and the $R$-group $R_\sigma$, if 
	\[Span_\mathbb{C}\Phi_\sigma\neq \mathfrak{a}_{M,\mathbb{C}}^\star, \]
	then one knows that, up to associated forms, 
	\[Span_\mathbb{C}\Phi_\sigma= \mathfrak{a}_{L,\mathbb{C}}^\star, \]
	for some parabolic subgroup $Q=LV$ with Levi $L\supset M$. Then by Theorem \ref{prod}, we know that 
	\[I^G_Q(\tau)\mbox{ is irreducible for all }\tau\in JH(I_{L\cap P}^L(\sigma)), \]
	which can not be discrete series. Therefore
	\[Span_\mathbb{C}\Phi_\sigma= \mathfrak{a}_{M,\mathbb{C}}^\star.\tag{SP} \]
	Which in turn says that the set of real parts of $\{\omega_\sigma\}_\sigma$ is finite as there are only finitely many rank-one reducibility hyperplanes in $\mathfrak{a}_M^\star$ by the {\bf Induction Assumption}. Indeed, we recently learned that such a claim (SP) is also a corollary of an old result of Harish-Chandra (cf. \cite[Theorem 5.4.5.7]{silberger2015introduction}, \cite[Theorem 3.9.1]{silberger1981discrete} or \cite[Corollary 8.7]{heiermann2004decomposition}).
	
	As for the set of imaginary parts of $\{\omega_\sigma\}_\sigma$, the finiteness follows from the facts that
	\begin{enumerate}[(i)]
		\item There are finitely many $R_\sigma$, and $R_\sigma$ is finite. This says that there are only finitely many linearly independent subsets.
		\item Each element in $R_\sigma$ is of finite order. This says that there are only finitely many solutions for each linearly independent subset.
	\end{enumerate}
Combining the finiteness of the real parts and imaginary parts, Clozel's finiteness conjecture holds under the {\bf Induction Assumption}.	
\end{proof} 
\begin{cor}
	Let $\tilde{G}$ be a finite central covering group of $G$, then Clozel's finiteness conjecture holds for $\tilde{G}$ under the {\bf Induction Assumption}. 
\end{cor}
\begin{proof}
	This follows from the Induction Assumption that there are only finitely many rank-one reducibility hyperplanes in $\mathfrak{a}_{M}^\star$ and the R-group theory in \cite{luo2017R}.
\end{proof}
\begin{rem}
	The Induction Assumption is a byproduct of the profound Langlands--Shahidi theory for generic $\sigma$s (cf. \cite{shahidi1990proof}). But for non-generic $\sigma$s, it follows from two conjectures of Shahidi (see \cite[Conjectures 9.2 and 9.4]{shahidi1990proof}). Indeed, for classical groups, those two conjectures are by-products of applying Arthur's (twisted) stable trace formula machine in his monumental book \cite{arthur2013endoscopic}. Moreover an explicit description of co-rank one reducibility points for classical groups is given in \cite{moeglin2003points,moeglin2002classification,moeglin2002construction,moeglin2007classification}).
\end{rem}
\begin{rem}
	Clozel's finiteness conjecture is also known for low rank groups of which their unitary duals are completely known (cf. \cite{hanzer2006unitary,hanzer2010unitary,konno2001induced,luo2018unitary,luo2018unitary2,matic2010new,muic1998unitary,sally1993induced,schoemann2014unitary}).
\end{rem}
%\begin{rem}
%	Note that in \cite{waldspurger2003formule,luo2017R}, we kind of resort to the generic irreducibility property of parabolic induction to get the rational property of $j_P(-)$ which implies the fact the there are only finitely many rank-one reducibility hyperplanes in $\mathfrak{a}_{M}^\star$. As the generic irreducibility property of parabolic induction follows easily from the product formula, i.e. Corollary \ref{prod1}. A natural question is that if one can prove directly the property that 
%	\[\mbox{there are only finitely many rank-one reducibility hyperplanes in $\mathfrak{a}_{M}^\star$}.\]
%\end{rem}

%		\begin{acknowledgements}
%		I am much indebted to Professor Wee Teck Gan for his guidance and numerous discussions on various topics. I would like to thank Professor Dipendra Prasad for helpful conversation. I also would like to thank Martin Raum for his kind help and support, and thank Maxim Gurevich for his seminar talk on a conjectural criterion of the irreducibility of parabolic inductions for $GL_n$ in the National University of Singapore which rekindles our enthusiasm to explore the mysterious internal structure of parabolic induction.
%		\end{acknowledgements}      
			
\bibliographystyle{spmpsci}
\bibliography{ref}

%\printbibliography
			
\end{document}